\documentclass{birkjour}
\usepackage{amssymb,latexsym,amsmath, amsfonts}
\usepackage{amsthm, indentfirst}

 \newtheorem{thm}{Theorem}[section]

 \theoremstyle{definition}
 \newtheorem{dfn}[thm]{Definition}
 \theoremstyle{remark}
 \newtheorem{rem}[thm]{Remark}
 
 \numberwithin{equation}{section}

\def\N{{\mathbb{N}}}

\def\R{{\mathbb{R}}}

\def\Cx{{\mathbb{C}}}

\def\tr{{\mathrm{tr}\,}}

\def\cM{{\mathcal{M}}}

\def\1{\mathbf{1}}

\newcommand{\smt}{ S(\cM,\tau)}

\begin{document}

%
%
%
%
%
%
%
%
%

\title{Dixmier-type traces on semifinite symmetric spaces}

\author{Galina Levitina}

\address{Mathematical Sciences Institute,
The Australian National University,
Canberra ACT 2601, Australia
}

\email{galina.levitina@anu.edu.au}

\thanks{The work of the first author was supported by the Australian Research Council. The work of the second author was partially supported by the Theoretical Physics and Mathematics Advancement
	Foundation ``BASIS''}
\author{Alexandr Usachev}
\address{School of Mathematics and Statistics,
	Central South University, Changsha 410075, People's Republic of China}
\email{dr.alex.usachev@gmail.com}


\date{January 1, 2004}
\dedicatory{To Fedor Sukochev on the occasion of his sixtieth birthday}

\begin{abstract}
We prove that a normalised linear functional on certain semifinite symmetric spaces respects tail majorisation if and only if it is a Dixmier-type trace.
\end{abstract}

\maketitle
\section{Introduction}

In 1966, J. Dixmier has discovered traces, vanishing on finite-rank operators \cite{D}. Precisely, for a positive convex increasing to infinity sequence $\{\psi(n)\}_{n=0}^\infty$ such that $\psi(2n)/\psi(n)\to 1$ as $n\to\infty$, he considered the Lorentz space
$$\mathcal M_\psi = \left\{A \ \text{is compact} : \sup_{n\ge0} \frac1{\psi(n+1)}\sum_{k=0}^n \mu(n,A) <\infty)\right\},$$
here $\{\mu(n,A)\}_{n=0}^\infty$ is the sequence of $s$-numbers of a compact operator $A$ on a separable Hilbert space. On the positive cone of this Lorentz space, Dixmier defined a functional
$${\rm Tr}_\omega(A)= \omega\left(\frac1{\psi(n+1)}\sum_{k=0}^n \mu(n,A)\right), \ 0\le A \in \mathcal M_\psi$$
and proved that it extends by linearity to a trace on the whole space, provided that $\omega$ is a dilation invariant extended limit on the space $\ell_\infty$ of all bounded sequences. Nowadays such traces are termed Dixmier traces.

In late 1980s, A. Connes employed Dixmier traces in his noncommutative geometry to construct an analogue of the integral in the noncommutative context \cite{C3, C}. Various problems in noncommutative geometry have stimulated the study of Dixmier traces. In particular, the development of semifinite noncommutative geometry unavoidably led to a semifinite counterpart of the Dixmier construction \cite{CPS2, CRSS}.

In 1998, F. Sukochev jointly with P. Dodds, B. de Pagter and E. Semenov studied fully symmetric functionals on fully symmetric spaces (Lorentz spaces are particular examples of them) \cite{DPSS}. It follows directly from the definition that Dixmier traces are fully symmetric functionals (see Definition \eqref{def_functionals} below). Theorem 3.4 of \cite{DPSS} establishes a criterion for existence of non-trivial symmetric functionals on Lorentz spaces.  
In 2011, this direction in the theory of singular traces has resulted in a joint paper of F. Sukochev with N. Kalton and A. Sedaev, showing that Dixmier traces are the only fully symmetric functionals on Lorentz spaces \cite{KSS}.

Fedor Sukochev and his co-authors were not the only group studying Dixmier (singular) traces at that period. Starting from late 80s, A. Pietsch has been developing the theory of traces in the general context of operators between Banach spaces. By mid 2010s this work gained its ultimate and elegant form \cite{P_trI, P_trII, P_trIII}. 

In 1990s, there was another direction in the study of singular traces -- the investigation of the traceability of operators \cite{Varga, GI95, AGPS}. The idea was to determine whether there exist traces taking non-zero value on a given compact operator. It turned out that this problem is equivalent to the existence of non-trivial singular traces on the principal ideal generated by this operator. 

In this note we consider a counterpart of Lorentz spaces - symmetric spaces $\mathcal I_h(\mathcal M, \tau)$, which are closed with respect to a so-called tail majorisation. As shown in \cite{CSZ}, tail majorisation is fundamental for characterisation of interpolation spaces in the the couple $(\mathcal L_0, \mathcal L_1)$. Following \cite{GI95, AGPS} we construct analogues of Dixmier traces on these spaces. Further, we prove that they are the only symmetric functionals on $\mathcal  I_h(\mathcal M, \tau)$ which respect tail majorisation. We present general semifinite version of this results.

\section{Preliminaries}

In this section, we recall main notions of the theory of noncommutative integration.

In what follows,  $H$ is a complex separable Hilbert space and $B(H)$ is the $C^*$-algebra of all bounded linear operators on $H$, equipped with the uniform norm $\|\cdot\|_\infty$, and
$\mathbf{1}$ is the identity operator on $H$.

For a self-adjoint operator $X$ in $H$, we denote by $E^X$ the spectral measure of $X$. For any closed densely-defined operator $X:\mathfrak{D}(X)\to H$ with the domain $\mathfrak{D}(X)$, the absolute value of $X$ is defined as $|X|=(X^*X)^{1/2}.$ 
%
%

%

For more details on von Neumann algebra
theory the reader is referred to e.g. \cite{Dix, KR1, KR2}
or \cite{Ta1}.

Throughout the present paper we assume that $\mathcal{M}$ is a von Neumann algebra, that is a  $*$-subalgebra  of $B(H)$ containing $\mathbf{1}$ and closed with respect to the weak operator topology. 

The collection of all (self-adjoint) projections in $\cM$ is denoted by $P(\cM)$. It is well-known (see e.g. \cite[Proposition 1.1]{Ta1}) that $P(\cM)$ is a complete lattice. The supremum (respectively, infimum) of two projection $P,Q\in P(\cM)$ is denoted by $P\vee Q$ (respectively, $P\wedge Q$). 
Two projections $P,Q\in P(\cM)$ are said to be equivalent (with respect to $\cM$), denoted by $P\sim Q$, if there exists a partial isometry $v\in\cM$ such that $P=v^*v$ and $Q=vv^*$.

A trace $\tau$ on a von Neumann algebra $\cM$ is a functional on the positive cone $\cM^+$ with values in $[0,\infty]$ which is additive, positively homogeneous and satisfies  
$$\tau(XX^*)=\tau(X^*X),\quad X\in \cM.$$
It is said that a trace $\tau:\cM^+\to [0,\infty
]$ is 
\begin{itemize}
	\item faithful if for all $X\in\cM^+$, $\tau(X)=0$ implies that $X=0$;
	\item semifinite if for every nonzero $X\in\cM^+$ there exists $0\leq Y\leq X$, such that $0<\tau(Y)<\infty$;
	\item normal if $\tau(\sup_{i\in I} X_i)=\sup_{i\in I}\tau(X_i)$ for every bounded increasing net $\{X_i\}_{i\in I}$ in $\cM^+$.
	\item finite if $\tau(X)<\infty$ for all $X\in\cM^+$.
\end{itemize}
The extension of a faithful normal semifinite trace $\tau$ on the whole algebra $\cM$ is denoted by $\tau$ too.

Throughout this paper, $\cM$ is assumed to be a semifinite von Neumann algebra equipped with a faithful normal semifinite trace $\tau$.
%

Next, we recall the basics of the theory of $\tau-$measurable operators and generalized singular value function. 
General facts concerning measurable operators may
be found in \cite{Ne, Se} (see also \cite[Chapter
IX]{Ta2}) and the upcoming book \cite{DPS}. For the convenience of the reader, some of the basic definitions are recalled here.

A linear operator $X:\mathfrak{D}\left( X\right) \rightarrow H $ is said to be {\it affiliated} with $\mathcal{M}$
if $YX\subseteq XY$ for all $Y\in \mathcal{M}^{\prime }$, where $\mathcal{M}^{\prime }$ is the commutant  of $\mathcal{M}$, defined as $\mathcal{M}^\prime=\{T\in B(H): TS=ST\quad  \forall S\in \mathcal{M}\}$. If $X$ is a self-adjoint operator affiliated with $\cM$, then $E^X(B)\in\cM$ for any Borel set $B\subset\R$. If $X$ is a closed densely defined operator affiliated with $\cM$, then the projections $n(X), r(X)$ and $s(X)$ belong to $\cM$ (see e.g. \cite[Chapter 2]{DPS}).

A closed
densely defined 
operator $X:\mathfrak{D}\left( X\right) \rightarrow H $ affiliated with $\mathcal{M}$ is termed
{\it measurable} with respect to $\mathcal{M}$ if there exists a
sequence $\left\{ P_n\right\}_{n=1}^{\infty}\subset P\left(\mathcal{M}\right)$, such
that $P_n\uparrow \mathbf{1}$, $P_n(H)\subseteq\mathfrak{D}\left(X\right) $
and $\mathbf{1}-P_n$ is a finite projection (with respect to $\mathcal{M}$)
for all $n$. It should be noted that the condition $P_{n}\left(
H\right) \subseteq \mathfrak{D}\left( X\right) $ implies that
$XP_{n}\in \mathcal{M}$. The collection of all measurable
operators with respect to $\mathcal{M}$ is denoted by $S\left(
\mathcal{M} \right)$. It is a unital $\ast $-algebra
with respect to the strong sums and products (denoted simply by $X+Y$ and $XY$ for all $X,Y\in S\left( \mathcal{M}\right) $) \cite[Corollary 5.2]{Se}.

An operator $X\in S\left( \mathcal{M}\right) $ is called $\tau-$measurable if there exists a sequence
$\left\{P_n\right\}_{n=1}^{\infty}$ in $P\left(\mathcal{M}\right)$ such that
$P_n\uparrow \mathbf{1},$ $P_n\left(H\right)\subseteq \mathfrak{D}\left(X\right)$ and
$\tau(\mathbf{1}-P_n)<\infty $ for all $n.$ The collection of all $\tau $-measurable
operators is a unital $\ast $-subalgebra of $S\left(
\mathcal{M}\right) $ and is denoted by $S\left( \mathcal{M}, \tau\right)
$ \cite[Theorem 4]{Ne}. It is well known that an operator $X$ belongs to $S\left(
\mathcal{M}, \tau\right) $ if and only if $X\in S(\mathcal{M})$
and there exists $\lambda>0$ such that $\tau(E^{|X|}(\lambda,
\infty))<\infty$ \cite[Theorem 2 (ii)]{Ne}. Alternatively, an operator $X$
affiliated with $\mathcal{M}$ is  $\tau$-measurable (see
\textbf{}\cite{FK}) if and only if
$$\tau\left(E^{|X|}\bigl(n,\infty\bigr)\right)\to 0,\quad n\to\infty.$$

%

The generalized singular value function $\mu(X):t\rightarrow \mu(t,X)$, $t>0$,  of an operator $X\in \smt$
is defined (see \cite[Definition 2.1]{FK}) by setting
\begin{equation}\label{def_mu}
	\mu(t,X)=\inf\{\|XP\|_\infty:\ P\in P(\mathcal{M}),\,  \tau(\mathbf{1}-P)\leq t\}.
\end{equation}

There exists an equivalent definition which involves the
distribution function of the operator $X$. For every
operator $X\in S(\mathcal{M},\tau),$ define the distribution function $d_X$ by setting
$$d_X(t)=\tau(E^{|X|}(t,\infty)),\quad t>0.$$
Then (see \cite[Proposition 2.2]{FK}) the singular value function $\mu(X)$ is the right-continuous inverse of $d_X(\cdot)$, that is 
$$\mu(t,X)=\inf\{s\geq0:\ d_{X}(s)\leq t\}.
$$

In particular, $\mu(X)$ is supported on $[0, \tau(\mathbf{1})]$.
For $X,Y\in S(\cM,\tau)$ we have \cite[Lemma 2.5 (v)]{FK}
$$	\mu(2t,X+Y)\leq \mu(t,X)+\mu(t,Y), \quad t>0.
$$
Note also that $\mu(0,X)=\lim_{t\downarrow 0}\mu(t,X)<\infty$ if and only if $X$ is bounded, in which case $\mu(0,X)=\|X\|_\infty$ (see e.g. \cite[Lemma 2.5]{FK}).

In the special case when $\mathcal{M}=L_\infty(0,\infty)$ is the von Neumann algebra of all
Lebesgue measurable essentially bounded functions on $(0,\infty)$ acting via multiplication on the Hilbert space
$\mathcal{H}=L_2(0,\infty)$, with the trace given by integration
with respect to Lebesgue measure $m$, the algebra $S(\cM,\tau)$ can be identified with the algebra
$$	S(0,\infty)=\{f \text{ is measurable:} \ \exists A\in\Sigma, m((0,\infty)\setminus A)<\infty, f\chi_A\in L_\infty(0,\infty)\},
$$
where $\chi_{A}$ denotes the characteristic function of a set $A\subset(0,\infty)$, and $\Sigma$ is the $\sigma$-algebra of Lebesgue measurable sets on $(0,\infty$).
In this case (see e.g. \cite[Section 2.3.2]{FK}), the singular value function $\mu(f)$ defined above is precisely the
decreasing rearrangement $f^*$ of the function $f\in S(0,\infty)$ given by
$$f^*(t)=\inf\{s\geq0:\ m(\{|f|> s\})\leq t\}.$$

If $\mathcal{M}=B(H)$ and $\tau$ is the
standard trace ${\rm tr}$, then it is not difficult to see that
$S(\mathcal{M})=S(\mathcal{M},\tau)=\mathcal{M}$ and 
for $X\in S(\mathcal{M},\tau)$ we have
$$\mu(t,X)=\mu(n,X),\quad t\in[n,n+1),\quad  n\geq0.$$
For a compact operator $X$ the sequence $\{\mu(n,X)\}_{n\geq0}$ is precisely the sequence of singular values of $X.$

The notion of singular value function naturally leads to various notions of majorisation of $\tau$-measurable operators.
For $0\leq A,B\in \smt$ we say that $A$ is majorised by $B$ (in the sense of Hardy-Littlewood) if for all  $0< t < \tau(\mathbf{1})$ we have
$$\int_0^t \mu(s,A)\,ds\leq \int_0^t \mu(s,B)\,ds.$$
In this case, we write $A\prec\prec B$. 

In the present paper we are mainly interested in another majorisation introduced in \cite{CSZ}.
\begin{dfn}For $0\leq A,B\in \smt$ we say that $A$ is \emph{tail majorised} by $B$ if for all $0< t < \tau(\mathbf{1})$ we have
	$$\int_t^{\tau(\mathbf{1})}\mu(s,A)\,ds\leq \int_t^{\tau(\mathbf{1})} \mu(s,B)\,ds.$$
	In this case, we shall write $A\prec\prec_{\rm tl}B$.
\end{dfn}

%
%

Following \cite[Section 2.4]{LSZ_book} we introduce the notion of Calkin spaces.

\begin{dfn}\label{Calkin_dfn}
	A linear subspace $E(\cM,\tau)$ of $\smt$ is called a \emph{Calkin space} if $B\in E(\cM,\tau)$, $A\in\smt$ and $\mu(A)\leq \mu(B)$ implies that $A\in E(\cM,\tau)$. In the special case when $\cM=L_\infty(0,b)$, $0< b \le \infty$ (respectively, $\cM=\ell_\infty(\N)$) we use the term Calkin function (respectively, sequence) space instead and denote $E(\cM,\tau)$ by $E(0,b)$ (respectively, $E(\N)$). 
\end{dfn}

A classical example of a Calkin space is the space $S_0(\cM,\tau)$ of all $\tau$-compact operators, defined as 
$$S_0(\cM,\tau) = \{X\in S(\cM,\tau) :  \ \mu(\infty, X)=\lim_{t\to\infty} \mu(t,X)=0\}.$$
Further examples of Calkin spaces are the noncommutative $L_p$-spaces, defined as 
$$L_p(\cM,\tau):=\{X\in \smt: \mu(X)\in L_p(0,\infty)\},\quad 0<p<\infty,$$
where $L_p(0,\infty)$ are the classical Lebesgue $L_p$-spaces.

A Calkin space is said to be fully symmetric if $A\prec\prec B\in E(\mathcal M,\tau)$ implies that $A\in E(\mathcal{M},\tau)$. It is straightforward that Lorentz spaces from Introduction are fully symmetric. We say that a Calkin space $E(\mathcal M,\tau)$ is closed with respect to the tail majorisation if $A\prec\prec_{\rm tl} B\in E(\mathcal M,\tau)$ implies that $A\in E(\mathcal{M},\tau)$.

Next we introduce the main objects of interest in the present paper, linear functionals that respect tails.  

\begin{dfn}\label{def_functionals}
	Let $E(\cM,\tau)$ be a Calkin space.
	\begin{itemize}
		\item A linear functional $\phi$ on $E(\cM,\tau)$ is said to be a \emph{symmetric functional}, if $\phi(X)=\phi(Y)$ for any $0\leq X,Y\in E(\cM,\tau)$ with $\mu(X)=\mu(Y).$
		\item A positive linear functional $\varphi$ on $E(\cM,\tau)$ is said to be \textit{fully symmetric}
			if $\varphi(A)\leq\varphi(B)$ provided that $0 \leq A, B \in E(\cM,\tau)$ are such $A \prec \prec B$.
		\item A positive linear functional $\varphi$ on $E(\cM,\tau)$ is said to be \textit{respecting tails}
		if $\varphi(A)\leq\varphi(B)$ provided that $0 \leq A, B \in E(\cM,\tau)$ are such $A \prec \prec_{\rm tl} B$.
	\end{itemize} 
\end{dfn}

	\begin{rem}\label{rem:TR_is_Sym}
		It is easy to see that every positive linear functional on $E(\cM,\tau)$ that respects tails is necessarily symmetric.
	\end{rem}
	
	On the space $L_\infty(0,b)$, $0< b \le \infty$ define the dilation operator
	$$(\sigma_{1/2} f)(t) = \begin{cases}
		f(2t), \ 0 < t < b/2\\
		0, t \ge b/2
	\end{cases}, \ f \in L_\infty(0,b).$$
	
	It follows from \cite[Proposition 2.3]{DPSS} that for every symmetric functional $\phi$ on a Calkin space $E(\cM,\tau)$ the following equality holds:
	\begin{equation}\label{SymInv}
		\phi(\sigma_{1/2} X) = \phi(X), \ X \in E(\cM,\tau),
	\end{equation}
	where $\sigma_{1/2} X$ stands for any operator such that $\mu(\sigma_{1/2} X)= \sigma_{1/2} \mu(X)$.

In the present paper we consider the following type of Calkin spaces.

\begin{dfn}\label{def_I_h}
	Let $\cM$ be a semifinite von Neumann algebra equipped with a faithful normal semifinite trace $\tau$. Let $\Omega$ denote the set of all convex decreasing functions $h$ on
	$(0,\infty)$. 
	For $h\in \Omega$, we define $\mathcal I_h(\cM,\tau)$ by setting
	$$\mathcal I_{h}(\mathcal M, \tau):= \{A \in \smt: \|A\|_{\mathcal I_h}= \sup_{t>0}  \frac{1}{h(t)} \int_t^{\tau(\mathbf{1})} \mu(s,A) \, ds < \infty\}.$$
\end{dfn}

It is clear that $\mathcal I_h(\cM,\tau)$ is a Calkin space, which is closed with respect to the tail majorisation.
At the same time, $\mathcal I_{h}(\mathcal M, \tau)$ is not fully symmetric. 
Indeed, for $g(t) = \frac2{(t+1)^2}$, $f(t) = \frac{\log(1+t)}{(t+1)^2}$, $t>e^2-1$ and $f(t)=\frac2{e^4}$, $0<t\le e^2-1$ we have $f \prec \prec g$. However,
setting $h(t)= \int_t^\infty g(s) ds$, we obtain 
$$\frac{1}{h(t)} \int_t^\infty f(s) \, ds = \frac{1}2(1+\log(1+t)) \to \infty, \ t\to \infty.$$
Thus, $f \notin \mathcal I_h(0,\infty)$.

 In the case when $h$ is a bounded function we have
$$\int_0^\infty \mu(s,A)\,ds\leq h(0)\sup_{t>0}  \frac{1}{h(t)} \int_t^\infty \mu(s,A) \, ds < \infty,$$
and therefore, $\mathcal{I}_h(\cM,\tau)\subset \mathcal L_1(\cM,\tau)$. In terminology of A. Piestch \cite[Section 3]{P_trII}  (used for principal ideals in the case $(\cM,\tau)=(B(H),\tr)$), this means that the Calkin space $\mathcal I_h(\cM,\tau)$  is of lower type.


Finally, we call a functional $\phi$ on $\mathcal I_h(\cM,\tau)$ normalised if $\phi(X) =1$ for every $X \in \smt$ such that $\mu(X) = - h'$.

\section{Main result}
In the present section we establish a form of a positive linear functional $\phi$ on $\mathcal{I}_h(\mathcal M,\tau)$, which respects tail majorisation.

We start with defining some properties of linear functionals on $L_\infty(0,b)$, $0< b \le \infty$.

\begin{dfn} \label{omega} A linear functional $\omega$ on $L_\infty(0,b)$, $0< b \le \infty$ is said to be 
	\begin{enumerate}
		\item[(i)] a \textit{state} if it is positive (i.e., $\omega(f) \ge 0$ if $f\ge0$) and normalised (in a sense that $\omega(\chi_{(0,b)})=1$);
		
		\item[(ii)] \textit{dilation invariant} if
		$\omega(f)=\omega(\sigma_{1/2} f)$ for every $f\in L_\infty(0,b)$;
		
		\item[(iii)] \textit{$h$-compatible}, for $h\in \Omega$, if 
		\begin{equation}\label{comp}
			\omega\left(t \mapsto \frac{h(2 t)}{h(t)}\right) =1.
		\end{equation}
	\end{enumerate}	
\end{dfn}

It can be shown that if the function $h$ satisfies the following condition:
\begin{equation}\label{limcond}
\lim_{t\to\infty} h(2t)/h(t)=1,
\end{equation}
then every dilation invariant state on $L_\infty(0,b)$ is $h$-compatible.
The following result constructs the analogue of Dixmier traces on $\mathcal{I}_h(\mathcal M,\tau)$. In the special case when $\mathcal M=B(H)$ (and $h\in \Omega$ satisfies condition \eqref{limcond}) this result was proved in \cite[Theorem 2.12]{AGPS}.

\begin{thm}\label{thm_tau_omega}Let $h\in \Omega$ and let $\mathcal{I}_h(\mathcal M,\tau)$ be as in Definition \ref{def_I_h}. 
	For every $h$-compatible dilation invariant state $\omega$ on $L_\infty(0,\tau(\1))$ the functional 
	$$\tau_\omega(A)= \omega\left(t \mapsto \frac{1}{h(t)}\int_t^{\tau(\1)} \mu(s,A)\,ds\right), \quad 0\le A\in \mathcal I_h(\mathcal M, \tau)$$ 
	extends to a normalised 
	linear functional on $\mathcal I_h(\mathcal M, \tau)$ that respects tails.
\end{thm}

\begin{proof}
It follows from \eqref{def_mu} that $\mu(\lambda A)=|\lambda|\mu(A)$ for any $A\in S(\cM,\tau)$ and any $\lambda\in \Cx$. Hence, $\tau_\omega$ is positively homogeneous. We shall show its additivity on the positive cone of $\mathcal I_{h}(\mathcal M, \tau)$. Let $0\le A,B\in \mathcal I_{h}(\mathcal M, \tau)$. By \cite[Theorem 3.3.3 and Theorem 3.3.4]{LSZ_book} for any $0\leq t\leq \tau(\1)$ we have that
	\begin{equation}\label{eq1}
		\int_0^{t} \mu(s,A+B)\,ds \le \int_0^t (\mu(s,A)+\mu(s,B))\, ds \le 2\int_0^t (\sigma_{1/2}\mu(A+B))(s)\, ds.
	\end{equation} 
	Since $A$ and $B$ are positive and since $\sigma_{1/2}\mu(s,T)=\mu(2s,T)=0$ for any $s>\frac{\tau(\1)}{2}$ and $T\in S(\cM,\tau)$, it follows that 
	\begin{align*}
		\int_0^{\tau(\1)} \mu(s,A+B)\,ds&=\|A+B\|_1=\tau(A+B)=\tau(A)+\tau(B)\\
		&=\int_0^{\tau(\1)}\mu(s,A)ds+\int_0^{\tau(\1)}\mu(s,B)ds,
	\end{align*}
	and
	\begin{align*}
		2\int_0^{\tau(\1)}\sigma_{1/2}\mu(s,A+B)ds&=2\int_0^{\tau(\1)/2}\mu(2s,A+B)ds\\
		&=\int_0^{\tau(\1)}\mu(s,A+B)ds=\|A+B\|_1,
	\end{align*}
	and so 
	\begin{align*}
		\int_0^{\tau(\1)} \mu(s,A+B)\,ds &= \int_0^{\tau(\1)} (\mu(s,A)+\mu(s,B))\, ds\\
		&=2\int_0^{\tau(\1)} (\sigma_{1/2}\mu(A+B))(s)\, ds.
	\end{align*}
	Combining the latter inequality with \eqref{eq1} we obtain that 
	\begin{align*}
		2\int_t^{\tau(\1)} (\sigma_{1/2}\mu(A+B))(s)\, ds &\le \int_t^{\tau(\1)} (\mu(s,A)+\mu(s,B))\, ds \\
		&\le \int_t^{\tau(\1)} \mu(s,A+B)\,ds
	\end{align*}
	for every $0\leq t\leq \tau(\1)$.
	
	Dividing all parts by $h(t)$ and applying $\omega$ we obtain
	\begin{align}\label{eq2}
		\begin{split}
			\omega\left(t \mapsto \frac{2}{h(t)}\int_t^{\tau(\1)} (\sigma_{1/2}\mu(A+B))(s)\, ds\right)\\
			 \le \omega\left(t \mapsto \frac{1}{h(t)}\int_t^{\tau(\1)} (\mu(s,A)+\mu(s,B))\, ds\right)\\ 
			 \le \omega\left(t \mapsto \frac{1}{h(t)}\int_t^{\tau(\1)} \mu(s,A+B)\,ds\right).
		\end{split}
	\end{align}

	Since $\mu(T)$ is supported on $[0,\tau(\1)]$ for any $T\in  S(\cM,\tau)$, it follows that  $\sigma_{1/2}\mu(T)$ is supported on $[0, \frac{\tau(\1)}{2}]$. Therefore, the function   $$t\mapsto \frac{2}{h(t)}\int_t^{\tau(\1)} (\sigma_{1/2}\mu(A+B))(s)\, ds$$ is supported on $[0, \frac{\tau(\1)}{2}]$. Hence, noting that $t\in [0, \frac{\tau(\1)}{2}]$, for the left-hand side of \eqref{eq2} we have 
	\begin{align*}
		\omega &\left(t \mapsto \frac{2}{h(t)}\int_t^{\tau(\1)} (\sigma_{1/2}\mu(A+B))(s)\, ds\right)\\
		& =\omega\left(t \mapsto \frac{1}{h(t)}\int_{2t}^{\tau(\1)} \mu(s,A+B) \, ds\right)\\
		& =
		\omega\left(t \mapsto \left(\frac{h(2t)}{h(t)}-1\right)\frac{1}{h(2t)}\int_{2t}^{\tau(\1)} \mu(s,A+B) \, ds\right)\\
		& \hspace{15mm} +\omega\left(t \mapsto \frac{1}{h(2t)}\int_{2t}^{\tau(\1)} \mu(s,A+B) \, ds\right)
	\end{align*}

	Using $h-$compatibility of $\omega$ we obtain
	$$ \omega\left(t \mapsto \left(\frac{h(2t)}{h(t)}-1\right)\right)=0$$
	and (since $\omega$ is positive)
	$$ \omega\left(t \mapsto \left(\frac{h(2t)}{h(t)}-1\right)\frac{1}{h(2t)}\int_{2t}^{\tau(\1)} \mu(s,A+B) \,ds\right)=0.$$
	Using dilation invariance of $\omega$ we conclude that
	\begin{align*}
		\omega & \left(t \mapsto \frac{2}{h(t)}\int_t^{\tau(\1)}(\sigma_{1/2}\mu(A+B))(s)\, ds\right)\\
		& \hspace{10mm} = \omega\left(t \mapsto \frac{1}{h(t)}\int_t^{\tau(\1)} \mu(s,A+B)\,ds\right).
	\end{align*}
	
	Referring to \eqref{eq2} we conclude that 
	$$\tau_{\omega}(A+B)\le\tau_{\omega}(A)+\tau_{\omega}(B)\le\tau_{\omega}(A+B),$$
	that is $\tau_\omega$ is linear on the positive cone of $\mathcal I_h(\mathcal M, \tau)$.
	Hence, it extends to a  linear functional on $\mathcal I_h(\mathcal M, \tau)$. The fact that the extended functional respects tails is straightforward.
\end{proof}

In the special case, when the von Neumann algebra is either atomless or atomic with atoms of equal trace we have the converse result that any functional on $\mathcal I_h(\mathcal M, \tau)$ which respect tails has the form $\tau_\omega$ as in Theorem \ref{thm_tau_omega}.

\begin{thm}\label{thm_main} Let $\mathcal M$ be a atomless (or atomic with atoms of equal trace)
	von Neumann algebra equipped with a faithful normal semifinite trace $\tau$. Let $h\in \Omega$ and let $\mathcal{I}_h(\mathcal M,\tau)$ be as in Definition \ref{def_I_h}. If $\varphi$ is a  normalized linear functional on $\mathcal I_h(\mathcal M, \tau)$ that respects tails, then 
	$$\varphi(A)=\tau_\omega= \omega\left(\frac{1}{h(t)}\int_t^{\tau(\1)} \mu(s,A)\,ds\right), \quad 0\le A\in \mathcal I_h(\mathcal M, \tau)$$
	for some $h$-compatible dilation invariant state $\omega$ on $L_\infty(0,{\tau(\1)})$.
\end{thm}

\begin{proof}
	By Remark \ref{rem:TR_is_Sym} any linear functional which respects tails is necessarily a symmetric functional. 
	Since $\cM$ is either atomless or atomic algebra with atoms of equal trace, it follows that there is a bijective correspondence between all symmetric functionals on $\mathcal I_h(\mathcal M, \tau)$ and all symmetric functionals on $\mathcal I_h(0, \tau(\1))$ \cite[Theorem 4.4.1]{LSZ_book}. Here, in the case of atomic algebra with atoms of equal trace, we isometrically embed the space $\ell_\infty(\N)$ into $L_\infty(0,\infty)$. Let $\varphi_0$ on $\mathcal I_h(0, \tau(\1))$ be the corresponding functional, that is a symmetric functional such that  $\varphi=\varphi_0 \circ \mu$.
	
	We firstly note that $\varphi_0$ respects tails. Indeed, let $f,g\in \mathcal I_h(0, \tau(\1))$ be such that $f\prec\prec_{\rm tl} g$. Then, there exist $A,B\in \smt$, such that $\mu(A)=f^*$ and $\mu(B)=g^*$. In particular, $A,B\in\mathcal I_h(\mathcal M, \tau)$ and $A\prec\prec_{\rm tl} B$. Since $\varphi$ respects tails, it follows that 
	\begin{align*}
		\varphi_0(f)&=\varphi_0(f^*)=\varphi_0(\mu(A))=\varphi(A)\leq \varphi(B)=\varphi_0(\mu(B))=\varphi_0(g),
	\end{align*}
	showing that $\varphi_0$ respect tails.

	Let $D_h$ be the linear span of all positive non-increasing functions from $\mathcal I_h(0, \tau(\1))$. Observe that $x\in D_h$ if and only if $x=y^*-z^*$, for $y,z \in \mathcal I_h(0, \tau(\1))$.	
	
	Define a map $T:D_h\to L_\infty(0,{\tau(\1)})$ by the formula
	$$(Tx)(t)=\frac1{h(t)}\int_t^{\tau(\1)} x(s)ds.$$
	Since $T$ is injection, one can define a linear functional $\omega_0$ on $TD_h$ by the formula $\omega_0(Tx)=\varphi_0(x).$
	
	Choose $f \in TD_h$ and, so, $x\in D_h$ such that $f(t)=\frac1{h(t)}\int_t^{\tau(\1)} x(s)ds.$ We have  
	$$ \int_t^{\tau(\1)} x(s)\,ds \le\left|\int_t^{\tau(\1)} x(s)\,ds\right| \le \|x\|_{\mathcal I_h} \cdot\int_t^{\tau(\1)} (-h'(s))\,ds.$$
	
	Since $x\in D_h$, then it is of the form $x=y^*-z^*$ for some $y,z \in \mathcal I_h(0, \tau(\1))$. Thus,
	$$ \int_t^{\tau(\1)} y^*(s)\,ds \le \int_t^{\tau(\1)} (-\|x\|_{\mathcal I_h}h'(s)+z^*(s))\,ds.$$
	
	Since $\varphi_0$ is normalized tail-respecting functional, it follows that
	$$\varphi_0(y^*)=\varphi_0(y)\le \|x\|_{\mathcal I_h} \varphi_0(-h')+ \varphi_0(z)=\|x\|_{\mathcal I_h} + \varphi_0(z^*)$$ and, 
	$\varphi_0(x)\le \|x\|_{\mathcal I_h}.$ So, $\omega_0(f)\le \|f\|_\infty$. By Hahn-Banach Theorem there exists $\omega$ on $L_\infty(0,{\tau(\1)})$ such that it coincides with $\omega_0$ on $TD_h$ and $\omega(f)\le \|f\|_\infty$ for every $f\in L_\infty(0,{\tau(\1)})$.
	
	Note, that the function $t\mapsto -2h'(2t)$ belongs to $D_h$. Due to \eqref{SymInv}, we have 
	$$\omega\left(t\mapsto \frac{h(2t)}{h(t)}\right)=\varphi_0(t\mapsto -2h'(2t))=\varphi_0(-h')=1.$$ 
	That is, $\omega$ is $h$-compatible. Hence,
	$\omega(|\frac{h(2t)}{h(t)}-1|)=0$.
	
	Let $R$ be the smallest $\sigma_{1/2}$-invariant subspace of $L_\infty(0,{\tau(\1)})$, containing $TD_h$. We claim that $\omega$ is $\sigma_{1/2}$-invariant on $R$. It is sufficient to show that $\omega(\sigma_{1/2} f)=\omega(f)$ for every $f\in TD_h$ or, equivalently, $\omega(\sigma_{1/2} Tx)=\omega(Tx)$ for every $x\in D_h$ or, (in view of definition of $\omega$ on $TD_h$ and  \eqref{SymInv}) equivalently, $\omega(\sigma_{1/2} Tx)=\omega(T(2\sigma_{1/2}x))$ for every $x\in D_h$.
	Since
	\begin{align*} |\sigma_{1/2} Tx-T(2\sigma_{1/2}x)| &=\left(1 - \frac{h(2t)}{h(t)}\right) \frac{1}{h(2t)}\int_{2t}^{\tau(\1)} x(u)\,du\\
		&\le  \left(1 - \frac{h(2t)}{h(t)}\right) \|x\|_{\mathcal I_h},
	\end{align*}
	we obtain the claim. Now, using the invariant form of the Hahn-Banach Theorem (see, e.g. \cite[Theorem 3.3.1]{Edwards}) we can extend $\omega$ (preserving the norm) to a $\sigma_{1/2}$-invariant functional on the whole $L_\infty(0,{\tau(\1)})$ (we shall denote it by the same letter).
	Since $\omega(\chi_{(0,{\tau(\1)})})=\omega_0(T(-h'))=1$, we conclude that $\omega$ is normalised.
	Thus, $\omega$ is a $h$-compatible dilation invariant state on $L_\infty(0,{\tau(\1)})$. 
\end{proof}

%

In conclusion we discuss existence of symmetric functionals  on $\mathcal I_h(\cM,\tau)$ which are supported at zero or at infinity. For that we firstly recall  necessary definitions.
	\begin{dfn}\label{def_singular}
	Let $E(\cM,\tau)$ be a Calkin space.
	\begin{itemize}
	
		\item A symmetric functional $\phi$ on $E(\cM,\tau)$ is said to be \textit{supported at infinity} if $\phi(X E^{|X|}(a,+\infty))=0$ for every $a>0$ and every $X \in E(\cM,\tau)$.
		\item A symmetric functional $\phi$ on $E(\cM,\tau)$ is said to be \textit{supported at zero} if $\phi(X E^{|X|}(0,a))=0$ for every $a>0$ and every $X \in E(\cM,\tau)$.
			\item A symmetric functional $\varphi$ on $E(\cM,\tau)$ is said to be \textit{singular} if $\phi(X)=0$ for any $X\in E(\cM,\tau)\cap \cM$ such that $\tau(s(X))<\infty$. 
		
	\end{itemize} 
\end{dfn}

	It follows from \cite[Theorem 4.3.2]{GI95} that every positive singular symmetric functional on a Calkin space has a unique decomposition as a sum of a positive symmetric functional supported at infinity and that supported at zero.
		
\begin{dfn} \label{omega_supp} (i) A linear functional $\omega$ on $L_\infty(0,b)$, $0< b \le \infty$ is said to be 
		\textit{supported at zero} if  $\omega(\chi_{(a,b)})=0$ for every $0<a\le b$. 
		
		(ii) A linear functional $\omega$ on $L_\infty(0,\infty)$ is said to be 
		\textit{supported at infinity} if  $\omega(\chi_{(0,a)})=0$ for every $0<a < \infty$.
\end{dfn}
	
It follows from Theorem \ref{thm_tau_omega} and \ref{thm_main} that a symmetric functional $\tau_\omega$ on $\mathcal I_h(\cM,\tau)$ which respect tails is supported at zero if and only if he corresponding state $\omega$ is supported at zero. If, in addition, the trace $\tau$ is infinite, these theorems imply that the functional $\tau_\omega$ is supported at infinity if and only if $\omega$ is supported at infinity. 

	The following result proves a criterion for the existence of tail-respecting functionals on spaces $\mathcal I_h(\cM,\tau)$.
	
	\begin{thm}\label{thm.symm_existence}Let $\mathcal M$ be a atomless (or atomic with atoms of equal trace)
	von Neumann algebra equipped with a faithful normal semifinite trace $\tau$ and let $h \in \Omega$.
	
	\begin{enumerate}
	\item The space $\mathcal I_h(\cM,\tau)$ admits non-zero tail-respecting functionals supported at infinity
		 if and only if  $\lim_{t \to \infty}  h(t) = 0$ and
		\begin{equation}\label{limcond1}
			\limsup_{t\to\infty} \frac{h(2 t)}{h(t)} =1.
		\end{equation}
		
		\item The space $\mathcal I_h(\cM,\tau)$ admits non-zero tail-respecting functionals supported at zero if and only if 
		$\lim_{t\to0} h(t) = \infty$ and 
		\begin{equation}\label{limcond2}
			\limsup_{t\to0} \frac{h(2 t)}{h(t)} =1.
		\end{equation}
\end{enumerate}

	\end{thm}
	
	\begin{proof} (i). Suppose firstly that the space $\mathcal I_h(\cM,\tau)$ admits  a non-trivial tail-respecting functional $\phi$ supported at infinity. Note that for $A\in \mathcal I_h(\cM,\tau)$ and every $t>0$ we have that $\int_t^{\infty} \mu(s,A)\,ds < \infty$ and, so, $\int_t^{\infty} \mu(s,A)\,ds \to 0$, as $t\to\infty$.  Therefore, if $\lim_{t \to \infty}  h(t) > 0$, then by Theorem \ref{thm_main} every tail-respecting functional on $\mathcal I_h(\cM,\tau)$ is zero. Thus, $h$ must vanish at infinity. To prove \eqref{limcond1} we firstly note that 
	 by Theorem \ref{thm_main} there exists an $h$-compatible dilation invariant state $\omega$ on $L_\infty(0,\infty)$, such that $\phi=\tau_\omega$. Moreover, this $\omega$ is supported at infinity. Hence, $\omega$ vanishes on every function with compact support. Since $\omega$ is a state, it follows that it vanishes on all functions vanishing at infinity. In other words, $\omega$ is an extended limit at infinity. Hence, $\omega(f) \le \limsup_{t\to\infty} f(t)$ for every $f\in L_\infty(0,\infty)$. In particular,
		$$1 = \omega\left(t \mapsto \frac{h(2 t)}{h(t)}\right)\le \limsup_{t\to\infty} \frac{h(2 t)}{h(t)} \le 1,$$ since $h$ is decreasing. Hence, $\limsup_{t\to\infty} \frac{h(2 t)}{h(t)} = 1$, as required. 
		
		Conversely, let $h$ be a convex decreasing function vanishing at infinity and let \eqref{limcond1} be satisfied. Then, the function $s \mapsto \frac{h(s t)}{h(t)}$ is convex and decreasing. Hence, the function 
		$$\alpha(s):=\limsup_{t\to\infty} \frac{h(s t)}{h(t)}, \ s>0$$
		is convex and decreasing. Since $\alpha(1)=1$ and $\alpha(2)=1$, the convexity implies that $\alpha(s)=1$ for every $s\ge1$. In particular,
		$$\limsup_{t\to\infty} \frac{h(2^n t)}{h(t)} =1 \ \text{for every} \ n\in \N.$$
		Let us choose a sequence $t_n \to \infty$, $n\to\infty$ such that 
		$$\lim_{n\to\infty} \frac{h(2^{n+1} t_n)}{h(t_n)} =1.$$
		Without loss of generality, suppose that $t_{n+1}> 2^{n+1} t_n$. For every $t_n \le t\le 2^n t_n$ we have
		$$1 \ge \frac{h(2 t)}{h(t)} \ge \frac{h(2^{n+1} t_n)}{h(t_n)} \to 1, n \to \infty.$$
		By \cite[Theorem 17]{SUZ1} there exists a dilation invariant extended limit $\omega$ (at infinity) on $L_\infty(0,\infty)$ such that
		$$\omega\left(t \mapsto \frac{h(2 t)}{h(t)}\right) = p_D\left(t \mapsto \frac{h(2 t)}{h(t)}\right) := \lim_{t \to \infty} \sup_{a \ge 1} \frac1{\log(t)} \int_a^{at} \frac{h(2 s)}{h(s)} \, \frac{ds}{s}.$$
		On the other hand, for every $n\in \N$ we can find a subsequence $\{s_n\} \subset \{t_n\}$ such that 
		$\frac{h(2 s)}{h(s)} \ge 1- 1/n$ for every $s_n \le s\le 2^n s_n$. We have
		$$1\ge p_D\left(t \mapsto \frac{h(2 t)}{h(t)}\right) = \lim_{n \to \infty}\frac1{\log (2^n)} \int_{s_n}^{2^n s_n} \frac{h(2 s)}{h(s)} \, \frac{ds}{s} \ge \lim_{n \to \infty}(1- 1/n) =1.$$
		Thus, there exists an $h$-compatible dilation invariant state $\omega$ on $L_\infty(0,\infty)$, supported at infinity. Theorem \ref{thm_tau_omega} implies that the space $\mathcal I_h(\cM,\tau)$ admits non-zero tail-respecting functionals supported at infinity.

		The proof of second assertion is similar and is, therefore, omitted.
	\end{proof}


\end{document}